\documentclass[11pt]{amsart}
\usepackage{amsmath}
\usepackage{amsfonts}
\usepackage{amsrefs}
\usepackage{xcolor}
\usepackage{amssymb}
\usepackage{hyperref}
\hypersetup{
    colorlinks=true,
    linkcolor=blue,
    filecolor=magenta,      
    urlcolor=cyan,
}

\newtheorem{thm}{Theorem}[section]

\newtheorem{cor}[thm]{Corollary}

\newtheorem{lemma}[thm]{Lemma}
\newtheorem{prop}[thm]{Proposition}

\theoremstyle{definition}
\newtheorem{defn}[thm]{Definition}
\newtheorem{remark}[thm]{Remark}

\newcommand{\bb}[1]{\mathbb{#1}}
\newcommand{\cl}[1]{\mathcal{#1}}

\newcommand{\rom}[1]{\mathrm{#1}}

\newcommand{\ep}{\varepsilon}

\newcommand{\FORAL}{\:\text{ for all }\:}

\newcommand{\qand}{\quad\text{and}\quad}
\newcommand{\qfor}{\quad\text{for}\quad}

\newcommand{\ip}[1]{\langle #1 \rangle}

\newcommand{\Bip}[1]{\Big\langle #1 \Big\rangle}

\newcommand{\Dim}{\operatorname{dim}}
\newcommand{\id}{{\operatorname{\textsl{id}\mskip2mu}}}

\newcommand{\re}{{\operatorname{Re}}}
\newcommand{\spn}{{\operatorname{span}}}

\renewcommand{\Re}{\mathop{\rm{Re}}\nolimits}
\newenvironment{sbmatrix}{\left[\begin{smallmatrix}}{\end{smallmatrix}\right]}

\newcommand{\OMIN}{{\operatorname{OMIN}}}
\newcommand{\OMAX}{{\operatorname{OMAX}}}

\DeclareMathOperator{\UCP}{UCP}
\newcommand{\UkP}{\ifmmode {\textnormal{U}k\textnormal{P}} \else {U$k$P}\fi}

\newcommand{\kmin}{\mathop{k\mbox{-}\hspace{-0.2ex}\mathrm{min}}}
\newcommand{\kmax}{\mathop{k\mbox{-}\hspace{-0.2ex}\mathrm{max}}}

\begin{document}

\title[Exactness and LLP]{Exactness and LLP Results \\via Operator System Methods}
\author[K.~R.~Davidson]{Kenneth R.~Davidson}
\address{Department of Pure Mathematics, University of Waterloo,
Waterloo, ON, Canada  N2L 3G1}
\email{krdavids@uwaterloo.ca}
\author[V.~I.~Paulsen]{Vern I.~Paulsen}
\address{Institute for Quantum Computing and Department of Pure Mathematics, University of Waterloo,
Waterloo, ON, Canada  N2L 3G1}
\email{vpaulsen@uwaterloo.ca}
\author[M.~Rahaman]{Mizanur Rahaman}
\address{Department of Mathematical Sciences and Wallenberg Centre for Quantum Technology, Chalmers University of Technology}
\email{mizanurr@chalmers.se}

\begin{abstract}

In this paper, we employ operator system techniques to investigate structural properties of C*-algebras. In particular, we provide more direct proofs of results concerning exactness and the local lifting property (LLP) of group C$^*$-algebras that avoid relying on the traditional heavy machinery of C$^*$-algebra theory. Briefly, these methods allow us to deduce that any C*-algebra containing $n$ unitaries whose {\it joint numerical radius}, in the sense defined by \cite{FKP},  is strictly less than $n$, must fail certain of these properties.
\end{abstract}

\maketitle

\section{Introduction}

The goal of these notes is to provide some, arguably, more direct proofs of results concerning exactness, the local lifting property(LLP), and hyperlinearity for discrete groups via operator system methods. In many ways the proofs in the works of Kirchberg and Wasserman introduced and relied upon many operator system ideas. But at that time there was not a developed tensor theory for operator systems. 
In the intervening years, many of their ideas were reworked using results and ideas from the tensor theory of operator spaces. 
One of the primary goals of the introduction of the tensor theory of operator systems was to see if it wasn't a more natural setting for many of the questions in C*-algebra theory and here we revisit Kirchberg's and Wasserman's works through this more developed theory in the operator system category. 

We begin by recalling some definitions and results from the operator system literature.

An element $p$ of an operator system is {\it strictly positive} if there is $\epsilon >0$ so that. $p - \epsilon 1$ is positive. 
A unital completely positive map $\phi: \cl S \to \cl T$ from one operator system onto another is called a {\it quotient map} provided that every strictly positive element of $\cl T$ has a positive pre-image in $\cl S$. 
The map is called a {\it complete quotient map} provided that $\phi_n: M_n(\cl S) \to M_n(\cl T)$ is a quotient map for all $n$.

An operator system is {\it exact} provided that whenever $\cl A$ is a unital C*-algebra and $J \subseteq \cl A$ is a two-sided ideal, the map 
$\cl S \otimes_{min} \cl A \to \cl S \otimes_{min} (\cl A/J)$ is a complete quotient map.

By a result of Kavruk, just as for C*-algebras, the property of being exact passes to subobjects:

\begin{prop}\cite[Proposition~4.10]{Ka1} \label{Kavruk-prop}
Let $\cl T$ be an exact operator system and let $\cl S \subseteq \cl T$ be a suboperator system.  If $\cl T$ is exact, then $\cl S$ is exact.
\end{prop}

An operator system $\cl S$ is said to have the {\it lifting property(LP)} if whenever $\phi: \cl S \to \cl A/J$ is a unital completely positive map(UCP), there exists a {\it UCP lift}, i.e., a UCP map $\psi: \cl S \to \cl A$ such that $\pi \circ \psi = \phi$ where $\pi: \cl A \to \cl A/J$ is the quotient *-homomorphism.  An operator system is said to have the {\it local lifting property(LLP)} if whenever $\phi: \cl S \to \cl A/J$ is UCP and $\cl T \subseteq \cl S$ is a finite dimensional operator subsystem, then the restriction of $\phi$ to $\cl T$ has a UCP lift.

 Kavruk \cite[Theorem~7.4]{Ka1} also proves that a finite dimensional operator system has the LP if and only if every map into the Calkin algebra, $Q(\ell^2)$, lifts, i.e., it is enough to replace arbitrary quotients by this particular quotient.

Given an operator system $\cl S$ the dual space $\cl S^d$ of bounded linear functionals has a natural family of positive cones in $M_n(\cl S^d)$, by defining $(f_{i,j}) \in M_n(\cl S^d)^+$ if and only if the map $\phi: \cl S \to M_n$ defined by
\[
 \phi(x) = (f_{i,j}(x)) \in M_n,
\]
is completely positive. When $\cl S$ is finite dimensional, Choi-Effros (corollary 4.5 in \cite{ChoiEff}) proved that there always exists a {\it strictly positive} linear functional, i.e., a functional $f: \cl S \to \bb C$ such that 
\[
 p \in \cl S^+, \,\, p \ne 0 \implies f(p) >0,
\]
and that any such functional satisfies the axioms to be an Archimedean order unit for $\cl S^d$, so that $(\cl S^d, f)$ satisfy the axioms to be an abstract operator system. Thus, the duals of finite dimensional operator systems are again operator systems, albeit non-canonically, since it depends on the choice of an order unit.

Fortunately, properties like exactness and LLP turn out to be independent of this choice of order unit.

Let $\cl A$ be a unital C$^*$-algebra and let $\cl S \subseteq \cl A$ be an operator system.  If the C*-algebra generated by the unitary elements in $\cl S$ is $\cl A$, then \cite{FKPT1} say that $\cl S$ {\it contains enough unitaries}.

It is convenient to introduce the following parameters associated to an operator system $\cl S$
\[
 r_k(\cl S) := \sup \{ \| \phi \|_{cb} : \phi: \cl T \to \cl S \,\,\, \UkP \},
\]
 and
 \[
  d_k(\cl S) := \sup \{ \| \phi \|_{cb} : \phi: \cl S \to \cl T \,\,\, \UkP \},
 \]
 where the supremum ranges over all operator systems $\cl T$ and all unital $k$-positive maps ( {\UkP}) maps $\phi$. It is helpful to think as $r$ and $d$ as standing for range and domain. Each sequence $(r_k(\cl S))_k$ and $(d_k(\cl S))_k$ is nonincreasing. For an operator system $\cl S$, we also consider the limits
\[
 r_{\infty} (\cl S) = \lim_{k \to \infty} r_k(\cl S) 
\]
and
\[
 d_{\infty} (\cl S) = \lim_{k \to \infty} d_k(\cl S). 
\]
We write down an important result that we will be using many time throughout the paper. 
 
\begin{thm} \label{exact-equivalences}
Let $\cl S \subseteq \cl A$ be a finite dimensional operator system that contains enough unitaries. Then the following are equivalent:
\begin{enumerate}
\item $\cl A$ is exact,
\item $\cl S$ is exact,
\item $r_{\infty}(\cl S) =1$,
\item $\cl S^d$ has the LP,
\item $d_{\infty}(\cl S^d) =1$.
\end{enumerate}
\end{thm}

\begin{proof} 
(1) implies (2) since subsystems of exact C*-algebras are exact by the  proposition \ref{Kavruk-prop}. 
(2) implies (1)  is \cite[Corollary~9.6]{KPTT1} or \cite[Proposition~10.12]{Ka1}.  
The equivalence of (2) and (3) and of (4) and (5) is in \cite{ADHPR}.  
The equivalence of (2) and (4) is \cite[Theorem~6.6]{Ka1}.
\end{proof}

Let $G$ be a finitely generated group with generators $u_1,\dots,u_n$ and let
\[ 
\cl S(\text{u}) = \spn \{ e, u_1,\dots,u_n, u_1^*,\dots, u_n^* \} \subseteq C^*(G),
\]
and 
\[
 \cl S_{\lambda}(\text{u}) = \spn \{ 1, \lambda(u_1),\dots, \lambda(u_n), \lambda(u_1)^*,\dots, \lambda(u_n)^* \} \subseteq C^*_{\lambda}(G).
\]
Then both of these sets contain enough unitaries, so the above theorem applies to both of these cases.

When $G = \bb F_n$ the free group on $n$ generators and we let $u_1,\dots,u_n \in \bb F_n$ then $\cl S(\text{u})$ is denoted $\cl S_n$.  If we let $\cl U_n \subseteq M_2 \oplus \dots M_2$($n$ copies) be the operator system spanned by the identity $I = I_2 \oplus \cdots \oplus I_2$ and the matrix units $E_{1,2,j}, E_{2,1,j}$ where the $j$ indicates that they are in the $j$-copy of $M_2$, then it is proven in \cite{FP} that $\cl U_n$ is completely order isomorphic to $\cl S_n^d$. 

Using the explicit form of the operator system dual, we can say the following:

\begin{thm}\label{F_n-exact-equivalences} The following are equivalent:
\begin{enumerate}
\item $C^*(\bb F_n)$ is exact,
\item $\cl S_n$ is exact, 
\item $r_{\infty}(\cl S_n) =1$
\item $\cl U_n$ has LP,
\item $d_{\infty}(\cl U_n) =1$,
\item if $X_1,\dots,X_n \in \cl A/J$ satisfy
\[ w(U_1 \otimes X_1 + \cdots U_n \otimes X_n) \le 1, \forall U_i \text{ unitary matrices },\]
then there exist $T_i \in \cl A, \, \pi(T_i) = X_i$ that satisfy the same inequality, where $\pi: \cl A \rightarrow \cl A/J$ is the quotient map.
\end{enumerate}
\end{thm}

\begin{proof} 
The equivalence of (1)--(5) follow from the above.  
The equivalence of (4) and (6) follows from the analysis of the operator system $\cl U_n$ in \cite{FKPT1}.  

In particular, they prove that a map $\phi: \cl U_n \to \cl A$ is UCP if and only if $\phi(1) = 1$ and 
\[
 w(\sum_k U_k \otimes \phi(E_{1,2,k})) \le 1/2.
\]
So if we set $X_k = 2 \phi(E_{1,2,k})$ we obtain (6).
\end{proof}

S. Wasserman\cite{Wa} proved that $C^*(\bb F_n)$ is not exact.  So by Wasserman's result we know that (6) fails. 
On the other hand, this theorem gives us a new means to show non-exactness.

One goal of this paper is to show directly, using an idea that goes back to Haagerup \cite{Ha}, that (6) fails, thus giving a new proof of Wasserman's result.
Moreover, similar numerical radius calculations will allow us to deduce that for any finitely generated non-abelian group $G$, 
either $G$ is not hyperlinear or $\cl S_{\lambda}(\text{u})$ fails to have the LP, where $\text{u}$ denotes the generating set. 
This result should be compared with a result of Ozawa \cite{Ozawa-QWEP} that either $C^*_{\lambda}(G)$ fails the local lifting property or fails to be QWEP. 
In contrast, Ozawa's result but uses a number of deep theorems.

Haagerup \cite{Ha} studied the analogous lifting problem to (6) but for the norm instead of the numerical radius. 
He proved that for any $n \ge 3$ there exists a quotient $\cl A/J$ and a completely contractive map $\phi: \ell^{\infty}_n \to \cl A/J$ 
such that any lifting $\psi: \ell^{\infty}_n \to \cl A$ has $\| \psi \|_{cb} \ge \frac{n}{2 \sqrt{n-1}}.$
If we let $e_i, 1 \le i \le n$ denote the canonical basis for $\ell^{\infty}_n$ and set $X_i = \phi(e_i)$, then these will satisfy the first part of (6), but for any lifting, setting $T_i = \psi(e_i)$, using the fact that $w(T) \ge \|T\|/2$ we will have that
\[
 \sup \{ w(U_1 \otimes T_1 + \cdots + U_n \otimes T_n): U_i \text{ unitary } \} \ge \|\psi\|_{cb}/2 \ge \frac{n}{4 \sqrt{n-1}} \ge 1,
\]
for any $n \ge 15$. Thus, using the above theorem, we see that Haagerup's calculations yield that $C^*(\bb F_n)$ is not exact for any $n \ge 15$. 
From this fact, standard techniques allow one to deduce that $C^*(\bb F_n)$ is not exact for all $n \ge 2$. 

Haagerup also showed that for $n=2$ any completely contractive mapping from $ \ell^{\infty}_2$ into a quotient has a completely contractive lifting. 
This is now known to be a consequence of the fact that the normed space $\ell^{\infty}_2$ has a unique operator space structure. 
For $n \ge 3$ it is known that $\ell^{\infty}_n$ has many different operator space structures. 
Since we know that (6) fails even for $n=2$, this shows that there is a distinctive difference between norm liftings and numerical radius liftings. 

In the next section we show that the failure of numerical radius liftings and the fact that $d_{\infty}(\cl U_n) >1$, and consequently that $C^*(\bb F_n)$ is not exact, are consequences of a classical result of Kesten combined with a $2\times2$ matrix trick.

\section{The Joint Numerical Radius}

Recall that given an element $T \in B(\cl H)$ the {\it numerical radius of T} is the quantity
\[
 w(T) = \sup \big\{ |\langle Th , h \rangle | :  h \in \cl H, \,\, \|h \|=1 \big\}.
\]
If $a$ is an element of a unital C*-algebra $\cl A$ and $\pi: \cl A \to B(\cl H)$ is any faithful *-homomorphism, then
\[ w(\pi(a)) = \sup \{ |s(a)| : s \text{ a state on } \cl A \},\]
and this quantity is denoted $w(a)$.

The following definition of the joint numerical radius of a set of elements was introduced in \cite{FKP}.

\begin{defn} 
Let $a_1,\dots,a_n$ be elements of a C*-algebra $\cl A$.  We set
\[
 w_k(a_1,\dots,a_n) = \sup \{ w(U_1 \otimes a_1 + \cdots + U_n \otimes a_n) : U_i \in M_k, \,\, \|U_i \| \le 1 \},
\]
and set
\[ 
w_{cb}(a_1,\dots,a_n) = \sup \{ w_k(a_1,\dots,a_n) : k \in \bb N \}.
\]
\end{defn}

We remark that the definition is unchanged if instead we assume that the $U_i$'s are all unitary matrices. 
In \cite{FKP}, $w_{cb}(a_1,\dots,a_n)$ is denoted as $w(a_1,\dots,a_n)$.

\begin{prop} \label{wcb=wsum} 
Let $G$ be a discrete group and let $g_1,\dots,g_n \in G$ and $\alpha_i \in \bb C$  Then
\begin{align*}
 w_{cb}(\alpha_1\lambda(g_1), \dots, \alpha_n\lambda(g_n)) &= w_1(\alpha_1 \lambda(g_1),\dots, \alpha_n \lambda(g_n)) \\
 &= w(|\alpha_1|\lambda(g_1)+ \cdots + |\alpha_n|\lambda(g_n)). 
\end{align*}
\end{prop}

\begin{proof} 
Let $\alpha_i = \omega_i |\alpha_i|$ with $|\omega_i| =1$ and setting $U_i= \overline{\omega_i}$, shows $w(|\alpha_1| \lambda(g_1) + \cdots + |\alpha_n| \lambda(g_n)) \le w_1( \alpha_1 \lambda(g_1), \ldots, \alpha_n \lambda(g_n)$.

Given contractions $U_1,\dots, U_n \in M_k$ and a unit vector $v=\sum_g v_g \delta_g \in \bb C^k \otimes \ell^2(G)$, we have that
\begin{align*} 
 \big| \big\langle \big( \sum_i U_i \otimes \alpha_i \lambda(g_i)\big) v , v \big\rangle \big| 
 &\le \sum_{i,g,h} |\alpha_i| |\langle U_i v_g , v_h \rangle | \langle \delta_{g_ig} , \delta_h \rangle \\&
 \le \sum_{i,g} |\alpha_i| \|v_g \| \| v_{g_ig} \| 
 = \big\langle \big( \sum_i |\alpha_i| \lambda(g_i) \big) u , u \big\rangle \\&
  \le w(|\alpha_1| \lambda(g_1) + \cdots + |\alpha_n| \lambda(g_n)), 
\end{align*}
where $u = \sum_g \|v_g\| \delta_g \in \ell^2(G),$
which shows 
\[
 w_{cb}(\alpha_1 \lambda(g_1),\dots, \alpha_n \lambda(g_n)) \le w(|\alpha_1| \lambda(g_1) + \cdots |\alpha_n| \lambda(g_n)).
 \qedhere
\]
\end{proof}

Recall that for an operator $A$ that $\Re A = (A+A^*)/2$ is its real part.

\begin{prop} \label{wcb=wre} 
Let $G$ be a discrete group and let $g_1,\dots,g_n \in G$.  
If $\alpha_i \ge0$, then for $A = \sum_{i=1}^n \alpha_i \lambda(g_i)$, 
\[
 w_{cb}( \alpha_1 \lambda(g_1),\dots, \alpha_n \lambda(g_n)) = w(A) = w(\Re A) = \| \Re A \| .
\]
\end{prop}

\begin{proof}
Observe that for a vector $v = \sum_g v_g \delta_g \in \ell^2(G)$, we have 
\[
 | \ip{ Av,v} | \le \ip{Au,u}
\]
where $u = \sum_g |v_g| \delta_g$. 
Because all of the coefficients are real,
\[
 \ip{Au,u} = \ip{u,A^*u} = \ip{A^*u,u} = \ip{ \Re A u,u}.
\]
Therefore with 
\[
 \Sigma = \big\{ u =  \sum_{g\in G} u_g \delta_g \in \ell^2(G),\ u_g \ge 0 ,\ \|u\|_2 = 1 \big\} ,
\]
we have
\begin{align*} 
 w(A) &= \sup_{u \in \Sigma} \ip{Au,u} = \sup_{u \in \Sigma} \ip{(\Re A ) u,u} 
 = w(\Re A) . 
\end{align*} 
Since $\Re A$ is self-adjoint, $w(\Re A) = \| \Re A \|$.
\end{proof}

\begin{prop} 
Let $g_1, \dots, g_n$ denote the standard generators of the free group $\bb F_n$. 
Then
\[
 w_{cb}(\lambda(g_1),\dots, \lambda(g_n)) =
 w\big( \sum_{i=1}^n \lambda(g_i) \big) = \sqrt{2n-1}.
\]
\end{prop}

\begin{proof}
It is a result of Kesten \cite{Kesten} that $\big\| \sum_{i=1}^n \lambda(g_i) + \lambda(g_i^{-1}) \big\| = 2\sqrt{2n-1}$.
Thus the result follows from Proposition~\ref{wcb=wre}. 
\end{proof}

\begin{lemma} \label{L:ucp from wcb}
Let $\cl A$ be a unital C*-algebra and let $x_1,\dots,x_n \in \cl A$. 
There is a UCP map $\phi: \cl U_n \to \cl A$ with $\phi(E_{1,2,i}) = x_i$ for $1 \le i \le n$
if and only if $w_{cb}(x_1,\dots,x_n) \le \frac12$,

Likewise the map $\phi$ is $k$-positive if and only if $w_k(x_1,\dots,x_n) \le \frac12$.
\end{lemma}

\begin{proof}
We may assume that $\cl A \subset \cl B(H)$. Assume that $w_{cb}(x_1,\dots,x_n) \le \frac12$,
Evidently the map $\phi$ is determined, as $\phi(I) = 1$ and $\phi(E_{2,1,i}) = x_i^*$.
Suppose that $p\ge1$ and 
\[
 X = I_p \otimes I + \sum_{i=1}^n A_i \otimes E_{1,2,i} + A_i^* \otimes E_{2,1,i} \ge 0  \quad\text{in } M_p \otimes \cl U_n .
\]
This holds precisely when
\[
\begin{bmatrix}I_p & A_i \\ A_i^* & I_p \end{bmatrix} \ge \begin{bmatrix}0_p & 0_p\\ 0_p & 0_p \end{bmatrix} \qfor 1 \le i \le n,
\]
which is equivalent to $\|A_i\| \le 1$ for $1 \le i \le n$. 
Then for any unit vector $\xi \in \bb C^p \otimes H$,
\begin{align*}
 \ip{ \phi_p(X) \xi, \xi } &= 1 + \Bip{ \big( \sum_{i=1}^n A_i \otimes x_i \big) \xi, \xi } + \Bip{ \xi, \big(\sum_{i=1}^n A_i \otimes x_i^*\big) \xi} \\&
 = 1 + 2 \re \Bip{ \big( \sum_{i=1}^n A_i \otimes x_i \big) \xi, \xi } \ge 1 - 1 = 0 .
\end{align*}
Thus $\phi_p(X) \ge 0$.

Now if $X = A_0 \otimes I + \sum_{i=1}^n A_i \otimes E_{1,2,i} + A_i^* \otimes E_{2,1,i} \ge 0$ in $M_p \otimes \cl U_n$ and $\ep > 0$,
then 
\[
 Y_\ep := (A_0 + \ep I_p)^{-1/2} ( X  + \ep I_p \otimes I ) (A_0 + \ep I_p)^{-1/2}
\]
has the form of the previous paragraph, and hence $\phi_p(Y_\ep) \ge 0$. Therefore
\[
 \phi_p(X  + \ep I_p \otimes I) = \big((A_0 + \ep I_p)^{1/2}\otimes I\big) \phi_p(Y_\ep) \big((A_0 + \ep I_p)^{1/2}\otimes I \ge 0\big).
\]
Letting $\ep \to 0$ shows that $\phi_p(X) \ge 0$. Therefore $\phi$ is completely positive.

Conversely, if $\phi$ is UCP, then the calculation in the first paragraph shows that 
\[
 \re \Bip{ \big( \sum_{i=1}^n A_i \otimes x_i \big) \xi, \xi } \ge -\frac12
\]
whenever $\|A_i\| \le 1$.
We can multiply each $A_i$ by a scalar of modulus 1 to deduce that
\[
 \Big|\Bip{ \big( \sum_{i=1}^n A_i \otimes x_i \big) \xi, \xi }\Big| \le \frac12.
\]
That is,  $w_p(x_1,\dots,x_n) \le \frac12$ for $p \ge 1$.

Now to consider the $k$-positivity of $\gamma$, it suffices to consider $p\le k$.
So as above, this comes down to the condition $w_k(x_1,\dots,x_n) \le \frac12$.
\end{proof}

We are now in a position to give a new proof of S. Wasserman's result.

\begin{thm} \label{newproofWasserman} 
For $n \ge 2$, 
\[
 d_{\infty}(\cl U_n) \ge \frac{n}{\sqrt{2n-1}}>1 ,
\]
and consequently, $C^*(\bb F_n)$ is not exact for $n \ge 2$.
\end{thm}

\begin{proof}
By the previous results, we have that
\[
 w_{cb}(\lambda(g_1),\dots, \lambda(g_n)) = w(\lambda(g_1)+ \cdots + \lambda(g_n) )= \sqrt{2n-1} =: w.
\]

Let $\cl R^{\omega}$ denote an ultrapower of the hyperfinite $II_1$-factor $\cl R$, and let 
\mbox{$\pi: \ell^{\infty}(\bb N, \cl R) \to \cl R^{\omega}$} denote the quotient map.  
Wassermann \cite{Wa1}*{Lemma, p.245} shows that the group von Neumann algebra $\cl L(\bb F_n)$ imbeds into $\cl R^\omega$
as a von Neumann subalgebra in a trace preserving manner. 

By Lemma~\ref{L:ucp from wcb}, we have a UCP map $\phi: \cl U_n \to \cl R^{\omega}$ with $\phi( E_{1,2,i}) = \frac{\lambda(g_i)}{2w}.$
Now by Kavruk \cite{Ka1} (see also \cite{Rob-Smith}), this map has a unital $k$-positive lifting $\psi: \cl U_n \to \ell^{\infty}(\cl R)$ that is self-adjoint.  
This map extends to a self-adjoint map on $M_2 \oplus \cdots \oplus M_2$ of the same cb-norm, which we still denote by $\psi$.  
Since $\ell^{\infty}(\cl R)$ is injective, let $\psi = \psi^+ - \psi^-$ be the Wittstock decomposition {\cite{Wittstock}} so that $\psi^{\pm}$ are CP and
\[ d_k(\cl U_n) \ge \| \psi \|_{cb} = \| \psi^+(1) + \psi^-(1)\|.\]

Now let $\phi^{\pm} = \pi \circ \psi^{\pm}$ where $\pi$ is the quotient map.  Let
\[ \begin{pmatrix} P_i^{\pm} & A_i^{\pm} \\ A_i^{\pm} & Q_i^{\pm} \end{pmatrix}\] 
be the Choi matrices of $\phi^{\pm}$ restricted to the i-th copy of $M_2$ so that
\[ A_i^+ - A_i^- = \lambda(g_i)/2w, \,\, \sum_i P_i^+ + Q_i^+ - P_i^- - Q_i^- = I.\]
Since
\[ \begin{pmatrix} P_i^- & - A_i^- \\ -A_i^- & Q_i^- \end{pmatrix} \ge 0,\]
we have that
\[ \begin{pmatrix} P_i^++ P_i^- & \lambda(g_i)/2w \\ \lambda(g_i)^*/2w & Q_i^+ + Q_i^- \end{pmatrix} \ge 0.\]
Conjugating by $-\lambda(g_i)$, applying the map $\id_2 \otimes \tau$ yields
\[ \begin{pmatrix} \tau(P_i^++P_i^-) & -1/2w \\ -1/2w. & \tau(Q_i^+ + Q_i^-) \end{pmatrix} \ge 0.\]
Using the fact that the sum of the entries of a positive matrix must be positive, it follows that
\[\tau(P_i^+ + P_i^- + Q_i^+ +Q_i^-) \ge 1/w.\]
Hence,
\[ d_k(\cl U_n) \ge \|\phi^+(I) + \phi^-(I) \| \ge \sum_i \tau(P_i^++P_i^- + Q_i^+ + Q_i^-) \ge \frac{n}{w}.\]
Since this is true for all k, we have that
\[
 d_{\infty}(\cl U_n) \ge \frac{n}{w} = \frac{n}{\sqrt{2n-1}}. \qedhere
\]
\end{proof}

\begin{remark} 
Kesten's result that for any group $G$ and symmetric susbset $S \subseteq G$ with $|S|=k$, $\| \sum_{s \in S} \lambda(g_s) \| \ge 2 \sqrt{k-1}$,  shows that for any group,
\[ w(\lambda(g_1)+ \cdots \lambda(g_n) ) \ge \sqrt{2n-1},\]
so we cannot improve the above lower bound by using other groups.
\end{remark}

The following now follows from another old result of Kesten \cite{Kesten2}.

\begin{thm}\label{T:nonamenable}
Let $G$ be a discrete group and let $g_1,\dots,g_n \in G$.  
Let $H$ be the subgroup generated by $\{ g_1,\dots,g_n \}$.
For $\alpha_i > 0$ and $A = \sum_{i=1}^n \alpha_i \lambda(g_i)$,  we have $w(A) = \sum_{i=1}^n \alpha_i$
if and only if $H$ is amenable.
\end{thm}

\begin{proof}
By Proposition \ref{wcb=wre} , $w(A) = \| \Re A \|$.
Kesten \cite{Kesten2} showed that $\| \Re A \| = \sum \alpha_i$ if and only if $H$ is amenable.
\end{proof}

Note that the above result tells us that $H$ is non-amenable if and only if for any such $A$,  $w(A) < \sum_{i=1}^n \alpha_i$ if and only if for all such $A$, $w(A) < \sum_{i=1}^n \alpha_i$.

\begin{cor} 
Let $G$ be a discrete group and let $g_1,\dots, g_n \in G$, let $H$ be the subgroup generated by $\{ g_1,\dots, g_n \}$ and let $\alpha_i \ne 0$. Then there exist $P_1,\dots, P_{n+1} \in B(\ell^2(G))$ such that
\[
 \begin{pmatrix} P_1 & \alpha_1 \lambda(g_1) & 0 & \cdots & 0 \\ 
 \overline{\alpha_1} \lambda(g_1)^* & P_2 & \alpha_2 \lambda(g_2) & \ddots & \vdots \\
 0 & \overline{\alpha_2} \lambda(g_2)^* & \ddots & \ddots & 0 \\
 \vdots & \ddots & \ddots & P_{n} & \alpha_n \lambda(g_n) \\
 0 & \cdots & 0 & \overline{\alpha_n} \lambda(g_n)^* & P_{n+1} \end{pmatrix} 
\]
is a positive operator on $\bb C^{n+1} \otimes \ell^2(G)$
with $\|P_1+ \cdots P_{n+1} \| < \sum_{i=1}^n |\alpha_i|$ if and only if $H$ is non-amenable.
\end{cor}

\begin{proof}
This follows from Proposition~\ref{wcb=wre}, Theorem~\ref{T:nonamenable} and the characterization in \cite[Corollary~3.5]{FKP} of $w_{cb}(\alpha_1 \lambda(g_1),\dots, \alpha_n \lambda(g_n))$ as the minimal norm of the sum of the diagonal entries of such a positive operator matrix.
\end{proof}

\begin{remark}
Theorem~\ref{T:nonamenable} can be seen directly using other well-known characterizations of amenable groups.
For simplicity, we assume that $H=G$.
Indeed, one observes that $w(A) = \sum_{i=1}^n \alpha_i$ if and only if there are unit vectors $u_k$
so that $\ip{ \lambda(g_i) u_k, u_k} \to 1$ for $1 \le i \le n$.
This in turn is easily seen to be the same as $\| \lambda(g_i) u_k - u_k \| \to 0$.
From this, we see that $\| \lambda(g) u_k - u_k \| \to 0$ for all $g \in G$.
This says that $C^*_\lambda(G)$ has almost invariant vectors, which is well-known to be equivalent to amenability, 
e.g.\ see \cite[Theorem 9.13.11]{DavidsonFAOA}. 
\end{remark}

\section{Property LLP, Quasidiagonality and reduced group C*-algebras}

Many authors have written about connections between the LLP and quasidiagonality for reduced group C*-algebras.
The books of Brown-Ozawa\cite{Brown-Ozawa} and G. Pisier\cite{Pi} are an excellent source for references and many of the results relating these concepts.  In this section, we use the ideas of the previous section, in particular the joint cb-numerical radius, to derive some related results.

\begin{thm} \label{T:fail LP or not hyperlinear}
Let $\cl A$ be a unital C*-algebra containing unitaries $u_1, \dots, u_n$ such that
$w_{cb}(u_1,\dots,u_n) <n$ and let $\cl S = \spn \{ 1, u_1, \dots, u_n, u_1^*,  \dots, u_n^* \}$.
If $\cl S$ has the LP, then $\cl A$ cannot be embedded into $\cl R^{\omega}$ 
or into $\frac{\Pi_{n\in \mathbb{N }} M_n(\mathbb{C}) }{\oplus_{n\in \mathbb{N}} M_n (\mathbb{C})}$.
\end{thm} 

\begin{proof} 
Set $w_{cb}(u_1,\dots,u_n) = w <n$. Let $\pi: \ell^{\infty}(\bb N, \cl R) \to \cl R^{\omega}$ be the quotient map.  Assume that there exists a unital isometric *-homomorphism $\rho: \cl A \to \cl R^{\omega}$.  We will prove that the restriction of $\rho$ to $\cl S$ does not have a UCP lifting.

We have a UCP map $\Phi: \cl U_n \to \cl S$ with $\Phi(E_{1,2,i}) = \frac{u_i}{2w}.$ 
Since $\cl S$ has the LP there is a UCP map $\Psi: \cl S \to \ell^{\infty}(\bb N, \cl R)$ such that $\pi \circ \Psi = \rho_{\cl S}$.

Since $\ell^{\infty}(\bb N, \cl R)$ is an injective C*-algebra, the map $\Psi \circ \Phi: \cl U_n \to \ell^{\infty}(\bb N, \cl R)$, extends to a UCP map 
$\Gamma: M_2 \oplus \cdots \oplus M_2 \to \ell^{\infty}(\bb N, \cl R)$ and the map $\pi \circ \Gamma$ is a UCP extension of $\rho \circ \Phi$.

If we set $P_j = \rho \circ \Phi(E_{1,1,j})$ and $Q_j = \rho \circ \Phi(E_{2,2,j})$, then the complete positivity of $\rho \circ \Phi$ on the $j$-th copy of $M_2$ implies that the matrix
\[
 \begin{pmatrix} P_j & \frac{\rho(u_j)}{2w} \\ \frac{\rho(u_j^*)}{2w} & Q_j \end{pmatrix}
\]
is positive.  
Conjugating by $\begin{pmatrix} -\rho(u_j)^* & 0 \\0 & 1 \end{pmatrix}$, taking the  trace of each element and summing the entries of the matrix,  we have that $\tau(P_j+Q_j) \ge \frac{1}{w}$ and hence,
\[
 1= \tau(\pi \circ \Gamma(I))= \sum_{j=1}^n \tau(P_j + Q_j) \ge \frac{n}{w} >1.
\]

This contradiction, proves that the map $\rho \circ \Phi$ cannot have a UCP extension to the direct sum of the matrix algebras, and hence cannot have a lift. 

The proof that $\cl A$ cannot be embedded into $\frac{\Pi_{n\in \mathbb{N }} M_n(\mathbb{C}) }{\oplus_{n\in \mathbb{N}} M_n (\mathbb{C})}$ is identical using a canonical trace on the quotient defined by taking a Banach generalized limit of the normalized traces on the matrix algebras, i.e.,
\[ \tau((A_1, A_2,\dots) + \oplus_{n \in \bb{N}} M_n(\bb C)) = \operatorname{glim}_n \tau_n(A_n),\]
where $\tau_n: M_n(\bb C) \to \bb C$ is the normalized trace and $\operatorname{glim}$ is a Banach generalized limit on $\ell^{\infty}(\bb N)$. 
\end{proof}

\begin{remark} 
Using the same method as in the proof of \ref{newproofWasserman} we have that if $\cl A$ can be embedded into $\cl R^{\omega}$ or 
into the quotient $\frac{\Pi_{n\in \mathbb{N }} M_n(\mathbb{C}) }{\oplus_{n\in \mathbb{N}} M_n (\mathbb{C})}$, then
\[ d_{\infty}( \cl S) \ge \frac{n}{w_{cb}(u_1,\dots,u_n)}.\]
\end{remark}
 
\begin{cor} 
Let $\cl A$ be a unital C*-algebra containing unitaries $u_1, \dots, u_n$ such that
$w_{cb}(u_1,\dots,u_n) <n$, then $\cl A$ is not quasidiagonal.
\end{cor}

\begin{proof} 
Recall that a C*-algebra is quasidiagonal if and only if there is a unital injective $*$-homomorphism 
$\cl A \rightarrow \frac{\Pi_{n\in \mathbb{N }} M_n(\mathbb{C}) }{\oplus_{n\in \mathbb{N}} M_n (\mathbb{C})}$ 
which admits a completely positive lift (see exercise 7.1.3 in \cite{Brown-Ozawa}). 
But by the above proof, the restriction of this inclusion to the operator system spanned by the unitaries does not lift.
\end{proof}

This leads to another proof of a result of J. Rosenberg \cite[Corollary~7.1.18]{Brown-Ozawa}:

\begin{cor} 
Let $G$ be a non-amenable discrete group, then $C_{\lambda}^*(G)$ is not quasidiagonal.
\end{cor}

\begin{proof} 
It is not hard to see that if $G$ is non-amenable, then some finitely generated subgroup $H$ must be non-amenable. 
Apply the last corollary to a generating set of $H$.
\end{proof}

\begin{cor} 
Let $n \ge 2$, and let 
\[
 \cl S_{\lambda}(n) := \spn \{ 1, \lambda(g_1), \dots, \lambda(g_n), \lambda(g_1)^*,\dots, \lambda(g_n)^* \} \subseteq C^*_{\lambda}(\bb F_n) .
\]
Then $\cl S_{\lambda}(n)$ fails to have the LP, and consequently $C^*_{\lambda}(\bb F_n)$ fails to have the LLP.
\end{cor}

\begin{proof} 
This follows by Theorem \ref{T:fail LP or not hyperlinear} since $C^*_{\lambda}(\bb F_n)$ does embed into $\cl R^{\omega}$.
\end{proof}

Recall that a group $G$ is {\it hyperlinear} if $C^*_{\lambda}(G)$ embeds into $\cl R^{\omega}$. 

\begin{cor} \label{non-amenable and LP}
Let $G$ be a finitely generated non-amenable group with generators $g_1,\dots, g_n$ and 
let $\cl S= \spn \{ 1, \lambda(g_1),\dots, \lambda(g_n), \lambda(g_1)^*,\dots, \lambda(g_n)^* \} $ in $C^*_{\lambda}(G).$
Then either $\cl S$ fails the LP or $G$ is not hyperlinear. 
Consequently, if $G$ is non-amenable and hyperlinear, then $C^*_{\lambda}(G)$ fails to have the LLP. 
\end{cor}

\begin{proof}
By Theorem~\ref{T:nonamenable}, $w(\lambda(g_1) + \cdots \lambda(g_n)) < n$.  
Hence, by Proposition~\ref{wcb=wsum}, we have $w_{cb}(\lambda(g_1),\dots, \lambda(g_n)) <n$. 
If $G$ is hyperlinear, then $\cl S$ fails the LP by Theorem~\ref{T:fail LP or not hyperlinear}.
\end{proof}

Another immediate consequence of Theorem~\ref{T:fail LP or not hyperlinear} is the following:

\begin{cor} \label{LP vs wcb in Romega}
Let $u_1,\dots,u_n$ be unitaries in $\cl R^\omega$, and let 
\[ \cl S = \spn \{ 1, u_1, \dots, u_n, u_1^*,  \dots, u_n^* \} .\]
If $w_{cb}(u_1,\dots,u_n) <n$, then $\cl S$ does not have LP.

If $\cl S$ is contained in any quasidiagonal C*-algebra and $w_{cb}(u_1,\dots,u_n) <n$, then $\cl S$ does not have LP.
\end{cor}

\section{Applications to Matrix Ranges}

In this section we use matrix ranges to give a lower bound on $d_{\infty}(\cl S_n)$ and obtain some interesting bounds on distances between certain matrix ranges.
Let $\rom T=(T_1,\dots,T_d), T_i \in B(\cl H_1)$ and  let
$\cl S_{\mathrm{T}}= \spn \{ I, T_1,\dots,T_d, T_1^*,\dots,T_d^* \}$.

The {\bf joint n-th matrix range} of $\rom T$ is the set of d-tuples of $n \times n$ matrices
\[
 W^n(\mathrm{T})= \{ (\phi(T_1),\dots,\phi(T_n)): \phi \in UCP(\cl S_{\mathrm{T}}, M_n) \}
\] 
and the {\bf joint matrix range}
is the disjoint union over $n$ of the n-th matrix ranges,  $\cl W(\mathrm{T})= \cup_{n \in \bb N} W^n(\mathrm{T})$.
 
For those more familiar with the concept, these matrix ranges are the most general bounded, matrix convex sets on $\bb C^d$.
 
We let 
\[
 \rom T^{\kmax}:=(T_1^{\kmax},\dots,T_d^{\kmax}) \qand \rom T^{\kmin}:=(T_1^{\kmin},\dots, T_d^{\kmin})
\] 
denote the images of the generators $T_i$ in the operator system $\OMAX_k(\cl S_{\rom T})$ and $\OMIN_k(\cl S_{\rom T})$ respectively. 
Since we have UCP maps, $\OMAX_k(\cl S_{\rom T}) \to \cl S_{\rom T} \to \OMIN_k(\cl S_{\rom T})$, sending generators to generators, we have
\[
 W^n(\rom T^{\kmin}) \subseteq W^n(\rom T) \subseteq W^n(\rom T^{\kmax}) \quad \text{for all } n.
\] 
There is equality of these sets for $n \le k$, since these three operator systems have the same UCP maps into $M_n$ for $n \le k$.

As pointed out in \cite{PP} these larger and smaller sets correspond to two canonical constructions in the theory of matrix convex sets.  
Indeed, for $n >k$, $W^n(\rom T^{\kmax})$ equals
\[  
 \big\{ (B_1,\dots, B_d): (\phi(B_1),\dots, \phi(B_d)) \in W^k(\rom T) \FORAL \phi \in \UCP(M_n, M_k) \big\},
\]
while
\[
 W^n(\rom T^{\kmin}) = \big\{ (B_1,\dots,B_d): B_i = \phi \circ \psi(T_i) \big\},
\]
where $\psi: \cl S_{\rom T} \to  \oplus_{i=1}^m M_k$ and $\phi: \oplus_{i=1}^m M_k \to M_n$ are both UCP, and $m$ is arbitrary. In words, $W^n(\rom T^{\kmax})$ consists of the d-tuples of matrices such that every $k \times k$ compression belongs to $W^k(\rom T)$, while $W^n(\rom T^{\kmin})$ consists of the d-tuples that can be built from sets of elements of $W^k(\rom T)$ in a matrix convex fashion.

It is rare that one can say how far apart these matrix convex sets are, especially asymptotically as $k \to +\infty$. The major contribution of \cite {PP} was to relate whether or not the distances between these sets tended to 0, to whether or not the operator system $\cl S_{\rom T}$ was exact or had the LP.

The purpose of this section is to describe these matrix ranges in the case of the operator systems 
\[
 \cl S_d= \spn \{ 1, g_1,\dots,g_d, g_1^*,\dots,g_d^* \} \subseteq C^*(\bb F^d),
\]
and 
\[
 \cl S_d^d \simeq \cl U_d= \spn \{ I, E_{1,2,i}, E_{2,1,i}: 1 \le i \le d \} \subseteq \oplus_{i=1}^d M_2,
\]
and to say as much about the distance between the matrix ranges as possible.
 
\begin{prop} We have:
\begin{enumerate}
\item $W^n(g_1,\dots,g_d) = \{ (B_1,\dots,B_d) :  \|B_i \| \le 1 \},$
\item $W^n(E_{1,2,1},\dots,E_{1,2,d}) = \{ (A_1,\dots,A_d):  w_{cb}(A_1,\dots,A_d) \le 1/2 \},$
\item $W^n(E_{1,2,1}^{\kmax}, \dots, E_{1,2,d}^{\kmax})= \{ (A_1,\dots,A_d): w_k(A_1,\dots,A_d) \le 1/2 \},$
\item $W^n(g_1^{\kmin}, \dots, g_n^{\kmin}) = $\vspace{-1.6ex}
\begin{align*}\strut\qquad
  \big\{ (B_1,\dots,B_d): w \big( \sum_{i=1}^d B_i \otimes A_i \big) \le 1/2 \text{ whenever } w_k(A_1,\dots,A_d) \le 1/2 \big\}.
\end{align*}
\end{enumerate}
\end{prop}

\begin{proof} 
The first statement is obvious.  
The second statement follows from Lemma~\ref{L:ucp from wcb}.
Likewise the third statement also follows from Lemma~\ref{L:ucp from wcb}.
 
To see the fourth statement note that 
\[
 Q:=I_p \otimes 1 + \sum_{i=1}^d A_i \otimes g_i^{\kmin} + \big( \sum_{i=1}^d A_i \otimes g_i^{\kmin} \big)^* \in M_p(\cl S_n^{\kmin})^+
\]
if and only if for every UCP $\phi: \cl S_n \to M_k,\, \phi(g_i) = B_i$, we have that
\[
 I_p \otimes I_k + \sum_{i=1}^d A_i \otimes B_i + \sum_{i=1}^d A_i^* \otimes B_i^* \ge 0.
\]
But such a map $\phi$ exists for every d-tuple of contractions in $M_k$; and so we see that $Q \in M_p(\cl S_n^{\kmin})^+$ 
if and only if $w_k(A_1,\dots,A_d) \le 1/2.$
Now (4) follows.  
\end{proof} 

\begin{remark} 
The above result can also be proven by using the fact that for any finite dimensional operator system, 
the matrix-ordered spaces and their matrix-ordered duals satisfy
\[
 \OMIN_k(\cl S)^d \simeq \OMAX_k(\cl S^d) \qand  \OMAX_k(\cl S)^d = \OMIN_k(\cl S^d) , 
\]
proving that the functional that is the support functional of the identity is an order unit, and chasing through the identifications.
\end{remark}

\begin{remark} 
We do not have useful explicit characterizations of the sets $W^n(g_1^{\kmax},\dots, g_n^{\kmax})$ and $W^n( E_{1,2,1}^{\kmin},\dots, E_{1,2,d}^{\kmin})$.
\end{remark}

In order to get a bound on the distance between these matrix ranges, we first consider the general case.

To this end let $\rom T= (T_1,\dots,T_d)$ be a tuple of operators in $B(\cl H_1)$
and $\mathrm{R}= (R_1,\dots, R_d) \in B(\cl H_2)$, set
$\cl S_{\mathrm{T}}= \spn \{ I, T_1,\dots,T_d, T_1^*,\dots,T_d^* \}$ and $\cl S_{\mathrm{R}} = \spn \{I, R_1,\dots,R_d, R_1^*, \dots,R_d^* \}$.

Given two sets $X,Y$ in a normed space we let
\[
 d_H(X,Y)=  \sup_{x \in X} \inf_{y \in Y} \|x - y\|, 
\]
denote the asymmetric Hausdorff distance. So  $X \subset Y$ implies $d_H(X,Y) =0$.

We define a norm on d-tuples of matrices by
\[
 \|(B_1,\dots,B_d) \| = \sum_i \|B_i\|,
\]
where $\|B_i\|$ is the usual C*-norm on matrices.

We define 
\[ d_H(\cl W(\mathrm{R}), \cl W(\mathrm{T})) = \sup_n d_H(W^n(\mathrm{R}), W^n( \mathrm{T})).\]

From this point on we assume that $\Dim(\cl S_{\mathrm{T}}) = \Dim(\cl S_{\mathrm{R}}) = 2d+1$ so that there is a well-defined linear map 
$\gamma: \cl S_{\mathrm{T}} \to \cl S_{\mathrm{R}}$ given by $\gamma(I) = I$, $\gamma(T_i) = R_i$ and $\gamma(T_i^*) = \gamma(R_i^*)$ for $1 \le i \le n$. 
Note that $\|\gamma\|_{cb} \ge \|\gamma(I)\| =1$.

\begin{prop} Let $\mathrm{T}, \mathrm{R}$ and $\gamma: \cl S_{\mathrm{T}} \to \cl S_{\mathrm{R}}$ be as above and assume that $\|T_i\|= \|R_i\|=1, 1 \le i \le d$. Define $\epsilon: \cl S_{\mathrm{T}} \to \ell^{\infty}_{2d+1}$ by sending each of the $2d+1$ basis elements to the basis elements, $e_i \in \ell^{\infty}_{2d+1}$.  Then
\[ \frac{\|\gamma\|_{cb} -1}{2 \|\epsilon\|_{cb}} \le d_H(\cl W(\mathrm{R}), \cl W(\mathrm{T})).\]
\end{prop}

\begin{proof}
We use the fact that $\|\gamma \|_{cb}$ is attained over self-adjoint elements \vspace{.25ex}
(to see this use that $X$ and $\begin{sbmatrix} 0 & X \\X^* & 0 \end{sbmatrix}$ have the same norm).  
Now the norm of a self-adjoint element in $M_n(\cl S_{\mathrm{T}})$ is of the form
\begin{multline*}
 \big\| A_0 \otimes I + \sum_{i=0}^{d}A_i \otimes T_i + \sum_{i=0}^d A_i^* \otimes T_i^* \big\| = \\
  \sup \big\{ \big\|A_0 \otimes I + \sum_{i=1}^d A_i \otimes B_i + \sum_{i=1}^d A^*_i \otimes B_i^* \big\| : (B_1,\dots, B_{n}) \in \cl W(\mathrm T) \big\}.
\end{multline*}
for some matrices $A_i \in M_n$ with $A_0=A_0^*$.

The map  $\epsilon: \cl S_{\mathrm T} \to \ell^{\infty}_{2n+1}$ given by $T_i \to e_i$ is completely bounded.
Thus if $A_i \in M_n, 0 \leq i \leq d$, then
\begin{align*}
 \max_{0 \leq i \leq d} \|A_i\| &= \big\| \epsilon \big( A_0 \otimes I+ \sum_{i=0}^d A_i \otimes T_i+ \sum_{i=1}^d A_i^* \otimes T_i^* \big) \big\| \\& 
 \leq \| \epsilon \|_{cb} \big\|A_0\otimes I+ \sum_{i=1}^d A_i \otimes T_i + \sum_{i=1}^d A_i^* \otimes T_i^* \big\| .
\end{align*}
Since $W({\mathrm R})$ and $\cl W({\mathrm T}))$ are compact, given any ${\mathrm B} \in W^n({\mathrm R})$, 
we can choose ${\mathrm C} \in W^n(\mathrm T)$ such that 
\[
 \sum_{i=1}^d \|B_i - C_i \| \leq d_H(\cl W({\mathrm R}), \cl W({\mathrm T})). 
\]

We have that
\begin{align*}
  \big\| A_0 &\otimes I_n + \sum_i A_i \otimes B_i+ \sum_i A_i^* \otimes B_i^* \big\| \\&
  \leq \big\| A_0 \otimes I_n + \sum_i A_i \otimes C_i + \sum_i A_i^* \otimes C_i^* \big\| \\&
  \quad + \big\| \sum_{i=1}^d A_i \otimes (B_i - C_i) + \sum_{i=1}^d A_i^* \otimes (B_i - C_i)^* \big\| \\&
  \leq \big\| A_0 \otimes I \!+\! \sum_i A_i \otimes T_i \!+\! \sum_i A_i^* \otimes T_i^* \big\| \!+\! 2 (\max_i \|A_i \|) d_H(\cl W({\mathrm R}), \cl W({\mathrm T})) \\&
  \leq \big(1 \!+\! 2\|\epsilon \|_{cb}\, d_H(\cl W({\mathrm R}), \cl W({\mathrm T})) \big) \ 
  \big\|A_0 \otimes I \!+\! \sum_i A_i \otimes T_i \!+\! \sum_i A_i^* \otimes T_i^* \big\|.
\end{align*}

Hence,  $\|\gamma\|_{cb} \le 1 + 2 \|\epsilon \|_{cb} d_H( \cl W({\mathrm R}), \cl W({\mathrm T}))$, so that
\[
 \frac{\|\gamma\|_{cb} -1}{2 \|\epsilon\|_{cb}} \le d_H ( \cl W({\mathrm R}), \cl W({\mathrm T})). \qedhere
\]
\end{proof}

\begin{thm} Let $E_{1,2,1},\dots, E_{1,2,n}$ be the canonical generators of $\cl U_n$.  We have that
\[
 d_H(\cl W(E_{1,2,1}^{\kmax},\dots,E_{1,2,n}^{\kmax}), \cl W(E_{1,2,1},\dots, E_{1,2,n})) \ge \frac{n- \sqrt{2n-1}}{2 \sqrt{2n-1}} .
\]
In particular, it is strictly greater than $\sqrt{\frac n 8} - \frac12$.
\end{thm}

\begin{proof}
We apply the above proposition to the case that $\cl S_T = \cl U_n$ with $T_i = E_{1,2,i}$ and
$\cl S_R = \OMAX_k(\cl U_n)$ with $R_i = T_i^{\kmax}$ the images of this basis and $\gamma_k: \cl U_n \to \OMAX_k(\cl U_n)$.  
We have shown that 
\[
 \| \gamma_k \|_{cb}= d_k(\cl U_n) \ge d_{\infty}(\cl U_n) \ge \frac{n}{\sqrt{2n-1}}. 
\]
Also,  
\begin{align*} 
 \big\| A_0 \otimes I + \sum_i A_i \otimes E_{1,2,i} &+ \sum_i A_i^* \otimes E_{1,2,i}^* \big\| = \\&
 =  \max_{i \ge 1} \big\| A_0 \otimes I + A_i \otimes E_{1,2,i} + A_i^* \otimes E_{1,2,i}^* \big\| \\&
 \ge  \max_{i \ge 0} \|A_i\| .
\end{align*}
Hence $ \|\epsilon\|_{cb} \le 1$.

Therefore, with $T_i = E_{1,2,i}$ for $1 \le i \le n$,
\begin{align*} 
 d_H( \cl W({\mathrm T}^{\kmax}), \cl W({\mathrm T})) &\ge \frac12 \Big(\frac{n}{\sqrt{2n-1}} - 1\Big) \\&
 =  \frac{n- \sqrt{2n-1}}{2\sqrt{2n-1}} > \sqrt{\frac n 8} - \frac12.
 \qedhere
\end{align*}
\end{proof}

\begin{cor} For every $n$ and $\epsilon >0$ there is a $p\ge1$ and matrices $A_1,\dots, A_n$ in $M_p$ such that
$w_k(A_1,\dots,A_n) =1$, and if $B_1,\dots, B_n \in M_p$ satisfy $w_{cb}(B_1,\dots, B_n) \le1$, then
\[
 \sum_{i=1}^n \|A_i - B_i \| \ge \frac{n- \sqrt{2n-1}}{\sqrt{2n-1}} - \epsilon.
\]
In particular, we can make this quantity greater than $\sqrt{\frac n 2} - 1$.
\end{cor}

\begin{proof} 
Recall that $W^p(E_{1,2,1}^{\kmax},\dots, E_{1,2,n}^{\kmax})$ is the n-tuples of matrices satisfying $w_k(A_1,\dots,A_n) \le 1/2$, 
which allows us to remove the 2 from the denominator.
\end{proof}

We now turn our attention to $\cl S_n$.

\begin{thm} For $n \ge 2$, we have that
\[ d_{\infty}(\cl S_n) \ge 1/2.\]
\end{thm}
\begin{proof}
By the previous result (ignoring the arbitrarily small $\epsilon$) we have:for all $k\ge1$, there is $p \in \bb N$ and $A_1,\dots,A_n \in M_p$ 
with $w_k(A_1,\dots,A_n) \le 1/2$ such that
\[
 \sum_i \|A_i - C_i \| \ge \frac{n -\sqrt{2n-1}}{2\sqrt{2n-1}} \ \ \FORAL \mathrm C\in M_p^n \text{ with } w_{cb}(C_1,\dots,C_n) \le \tfrac12.
\]

We claim that $P= I_p \otimes 1 + \sum_i A_i \otimes g_i + \sum_i A_i^* \otimes g_i^* \in M_p(\OMIN_k(\cl S_n))^+$.
This will be true if and only if for every UCP $\gamma: \cl S_n \to M_k$ one has $id_p \otimes \gamma(P) \ge 0.$
But $\gamma$ is UCP iff $\gamma(1) = I_k$ and $\gamma(g_i) = B_i$ with $B_i \in M_k$ and $\|B_i \| \le 1$.  But for such a set of B's, we have that $w(\sum_i B_i \otimes A_i) \le 1/2$ and hence $id_p \otimes \gamma(P) \ge 0$.

Thus,   $P \in M_p(\OMIN_k(\cl S_n))^+$ as claimed.
Set $T = \sum_i A_i \otimes g_i$, we have shown that $I_p \otimes 1 + T + T^* \ge 0.$
Note that the same proof shows that $I_p \otimes 1 + e^{i \theta} T + e^{-i \theta} T^* \ge 0$ so that $w(T) \le 1/2.$
Hence,
\[ \frac{\|T\|_{\OMIN_k(\cl S_n)}}{2} \le w(T) \le 1/2 \implies \|T\|_{\OMIN_k(\cl S_n)} \le 1.\]

Now let. $r= \| T \|_{\cl S_n}$.  Then $2r(I_p \otimes 1) + T + T^* \in M_p(\cl S_n)^+$ which implies that
$w_{cb}(A_1/2r,\dots,A_n/2r) \le 1/2$. Setting $C_i = A_i/2r$ we have that
\[ \sum_i \|A_i - A_i/2r\| \ge \frac{n -\sqrt{2n-1}}{2\sqrt{2n-1}}. \]
Thus,
\[ (1-1/2r) \sum_i \|A_i\| \ge  \frac{n -\sqrt{2n-1}}{2\sqrt{2n-1}}:=c_n.\]

This implies that  
\[
 d_{\infty}(\cl S_n) \ge r \ge \frac{\sum_i \|A_i\|}{2(\sum_i \|A_i\| - c_n)} \ge \frac{\sum_i \|A_i\|}{2(\sum_i \|A_i\| - c_n)} .
\]
  
  Note that $\frac{t}{2(t- c_n)}$ is a decreasing function of $t$ and $t= \sum_i \|A_i\| \le n$.
  Hence,
  \[
   d_{\infty}(\cl S_n) \ge \frac{n}{2n- c_n} =\frac{2n\sqrt{2n-1}}{(4n+1)\sqrt{2n-1} -n} \ge \frac12.
   \qedhere
  \]
  \end{proof}
 
 \begin{remark} We conjecture that $d_{\infty}(\cl S_n)$ is on the order of $\sqrt{n}.$
 \end{remark}

 \section{Acknowledgements}
 We thank Mehrdad Kalantar, Tatiana Shulman and Ebrahim Samei for helpful discussions. 
 


\begin{thebibliography}{99}

\bib{AO}{article}{
   author={Akemann, Charles A.},
   author={Ostrand, Phillip A.},
   title={Computing norms in group $C\sp*$-algebras},
   journal={Amer. J. Math.},
   volume={98},
   date={1976},
   pages={1015--1047},
}

\bib{ADHPR}{article}{
   author={Aubrun, Guillaume},
   author={Davidson, Kenneth R.},
   author={M\"uller-Hermes, Alexander},
   author={Paulsen, Vern I.},
   author={Rahaman, Mizanur},
   title={Completely bounded norms of $k$-positive maps},
   journal={J. Lond. Math. Soc. (2)},
   volume={109},
   date={2024},
   pages={Paper No. e12936, 21},
}

\bib{Brown}{article}{
   author={Brown, Nathanial P.},
   title={On quasidiagonal $C^*$-algebras},
   conference={
      title={Operator algebras and applications},
   },
   book={
      series={Adv. Stud. Pure Math.},
      volume={38},
      publisher={Math. Soc. Japan, Tokyo},
   },
   date={2004},
   pages={19--64},
 }
 
\bib{Brown-Ozawa}{book}{
   author={Brown, Nathanial P.},
   author={Ozawa, Narutaka},
   title={$C^*$-algebras and finite-dimensional approximations},
   series={Graduate Studies in Mathematics},
   volume={88},
   publisher={American Mathematical Society, Providence, RI},
   date={2008},
   pages={xvi+509},
}

\bib{ChoiEff}{article}{
author = {Choi, Man-Duen},
author = {Effros, Edward G},
title = {Injectivity and operator spaces},
journal = {Journal of Functional Analysis},
volume = {24},
pages = {156-209},
year = {1977},
}

\bib{DavidsonFAOA}{book}{
   author={Davidson, Kenneth R.},
   title={Functional analysis and operator algebras},
   series={CMS/CAIMS Books in Mathematics},
   volume={13},
   publisher={Springer, Cham},
   date={[2025] \copyright 2025},
   pages={xiv+797},
}

\bib{E-S}{article}{
   author={Eckhardt, Caleb},
   author={Shulman, Tatiana},
   title={On amenable Hilbert-Schmidt stable groups},
   journal={J. Funct. Anal.},
   volume={285},
   date={2023},
   pages={Paper No. 109954, 31},
}

\bib{FP}{article}{
   author={Farenick, Douglas},
   author={Paulsen, Vern I.},
   title={Operator system quotients of matrix algebras and their tensor
   products},
   journal={Math. Scand.},
   volume={111},
   date={2012},
   pages={210--243},
}

\bib{FKP}{article}{
author={Farenick, Douglas},
author={Kavruk, Ali S.},
author={Paulsen, Vern I.},
title={C*-algebras with the WEP and a multivariable analogue of Ando's theorem on the numerical radius},
journal={J. Operator Theory},
volume={70},
date={2013},
pages={573--590},
}

\bib{FKPT1}{article}{
   author={Farenick, Douglas},
   author={Kavruk, Ali S.},
   author={Paulsen, Vern I.},
   author={Todorov, Ivan G.},
   title={Operator systems from discrete groups},
   journal={Comm. Math. Phys.},
   volume={329},
   date={2014},
   pages={207--238},
}

\bib{FKPT2}{article}{
   author={Farenick, Douglas},
   author={Kavruk, Ali S.},
   author={Paulsen, Vern I.},
   author={Todorov, Ivan G.},
   title={Characterisations of the weak expectation property},
   journal={New York J. Math.},
   volume={24A},
   date={2018},
   pages={107--135},
}

\bib{Ha}{article}{
   author={Haagerup, Uffe},
   title={Injectivity and decomposition of completely bounded maps},
   conference={
      title={Operator algebras and their connections with topology and
      ergodic theory},
      address={Bu\c steni},
      date={1983},
   },
   book={
      series={Lecture Notes in Math.},
      volume={1132},
      publisher={Springer, Berlin},
   },
   date={1985},
   pages={170--222},
}

\bib{ISM}{article}{
   author={Ioana, Adrian},
   author={Spaas, Pieter},
   author={Wiersma, Matthew},
   title={Cohomological obstructions to lifting properties for full $\rm
   C^*$-algebras of property (T) groups},
   journal={Geom. Funct. Anal.},
   volume={30},
   date={2020},
   pages={1402--1438},
}

\bib{Haag-Thor}{article}{
   author={Haagerup, Uffe},
   author={Thorbj\o rnsen, Steen},
   title={A new application of random matrices: ${\rm Ext}(C^*_{\rm
   red}(F_2))$ is not a group},
   journal={Ann. of Math. (2)},
   volume={162},
   date={2005},
   pages={711--775},
}

\bib{Ka1}{article}{
   author={Kavruk, Ali \c S.},
   title={Nuclearity related properties in operator systems},
   journal={J. Operator Theory},
   volume={71},
   date={2014},
   pages={95--156},
}

\bib{KPTT1}{article}{
   author={Kavruk, Ali S.},
   author={Paulsen, Vern I.},
   author={Todorov, Ivan G.},
   author={Tomforde, Mark},
   title={Quotients, exactness, and nuclearity in the operator system category},
   journal={Adv. Math.},
   volume={235},
   date={2013},
   pages={321--360},
}

\bib{Kesten}{article}{
   author={Kesten, Harry},
   title={Symmetric random walks on groups},
   journal={Trans. Amer. Math. Soc.},
   volume={92},
   date={1959},
   pages={336--354},
}

\bib{Kesten2}{article}{
   author={Kesten, Harry},
   title={Full Banach mean values on countable groups},
   journal={Math. Scand.},
   date={1959},
   pages={146--156},
}

\bib{Ozawa-QWEP}{article}{
   author={Ozawa, Narutaka},
   title={About the QWEP conjecture},
   journal={Internat. J. Math.},
   volume={15},
   date={2004},
   pages={501--530},
}

\bib{PP}{article}{
   author={Passer, Benjamin},
   author={Paulsen, Vern I.},
   title={Matrix range characterizations of operator system properties},
   journal={J. Operator Theory},
   volume={85},
   date={2021},
   pages={547--568},
}

\bib{Pi}{book}{
author={Pisier, Gilles},
title={Tensor Products of C*-algebras and Operator Spaces: The Connes-Kirchberg Problem},
series={London Mathematical Society Student Series},
volume={96},
publisher={London Mathematical Society},
date={2020},
pages={xiv+484},
}

\bib{Rob-Smith}{article}{
   author={Robertson, A. G.},
   author={Smith, R. R.},
   title={Liftings and extensions of maps on $C^*$-algebras},
   journal={J. Operator Theory},
   volume={21},
   date={1989},
   pages={117--131},
}

\bib{Wa1}{article}{
   author={Wassermann, Simon},
   title={On tensor products of certain group $C\sp{\ast} $-algebras},
   journal={J. Functional Analysis},
   volume={23},
   date={1976},
   pages={239--254},
}

\bib{Wa}{article}{
   author={Wassermann, Simon},
   title={Tensor products of free-group $C^*$-algebras},
   journal={Bull. London Math. Soc.},
   volume={22},
   date={1990},
   pages={375--380},
}

\bib{Wittstock}{article}{
   author={Wittstock, Gerd},
   title={Ein operatorwertiger Hahn-Banach Satz},
   language={German, with English summary},
   journal={J. Functional Analysis},
   volume={40},
   date={1981},
   pages={127--150},
}
\end{thebibliography}
\end{document}